\documentclass[11pt]{article}

\usepackage{amsmath}
\usepackage{amsthm}
\usepackage{amsfonts}
\usepackage{amssymb}
\usepackage{latexsym}

\usepackage[noautoscale]{youngtab}

\usepackage{fullpage}

\usepackage{color, colortbl}
\usepackage{array}
\usepackage{multirow}

\usepackage{tikz}

\numberwithin{equation}{section}

\newtheorem{theorem}{Theorem}[section]

\newtheorem{lemma}[theorem]{Lemma}
\newtheorem{corollary}[theorem]{Corollary}
\newtheorem{conjecture}[theorem]{Conjecture}
{
\theoremstyle{definition}

\newtheorem{remark}[theorem]{Remark}
}

\DeclareMathOperator{\Av}{Av}

\usepackage[colorlinks=true,
linkcolor=webgreen,
filecolor=webbrown,
citecolor=webgreen]{hyperref}

\definecolor{webgreen}{rgb}{0,.5,0}
\definecolor{webbrown}{rgb}{.6,0,0}

\usepackage[symbol]{footmisc}

\newcommand{\oeis}[1]{\href{https://oeis.org/#1}{\textcolor{green!40!black}{\text{#1}}}}

\begin{document}

\title{On (shape-)Wilf-equivalence of certain sets of (partially ordered) patterns}

\author{Alexander Burstein\footnote{Department of Mathematics, Howard University, Washington, DC 20059, USA. Email: aburstein@howard.edu.}, Tian Han\footnote{College of Mathematical Science, Tianjin Normal University, Tianjin 300387, P. R. China. Email: hantian.hhpt@qq.com.}, Sergey Kitaev\footnote{Department of Mathematics and Statistics, University of Strathclyde, 26 Richmond Street, Glasgow G1 1XH, United Kingdom. Email: sergey.kitaev@strath.ac.uk.}\ \ and Philip Zhang\footnote{College of Mathematical Science, Tianjin Normal University, Tianjin 300387, P. R. China. Email: zhang@tjnu.edu.cn.}}

\date{May 23, 2024}

\maketitle

\noindent\textbf{Abstract.}  We prove a conjecture of Gao and Kitaev on Wilf-equivalence of sets of patterns $\{12345,12354\}$ and $\{45123,45213\}$ that extends the list of 10 related conjectures proved in the literature in a series of papers.  To achieve our goals, we prove generalized versions of shape-\!Wilf-equivalence results of Backelin, West, and Xin and use a particular result on shape-\!Wilf-equivalence of monotone patterns. We also derive general results on shape-\!Wilf-equivalence of certain classes of partially ordered patterns and use their specialization (also appearing in a paper by Bloom and Elizalde) as an essential piece in proving the conjecture. Our results allow us to show (shape-)Wilf-equivalence of large classes of sets of patterns, including 11 out of 12 classes found by Bean et al. in relation to the conjecture.\\

\noindent \textbf{AMS Classification 2010:}  05A05; 05A15; 05A19

\noindent \textbf{Keywords:}  permutation pattern, Wilf-equivalence, shape-\!Wilf-equivalence, partially ordered pattern

\section{Introduction} \label{sec:intro}

A permutation of length $n$ is a rearrangement of the set $[n]:=\{1,2,\ldots,n\}$. Denote by $S_n$ the set of permutations of $[n]$.  An occurrence of a (classical) permutation pattern $p=p_1\cdots p_k$ in a permutation $\pi=\pi_1\cdots\pi_n$ is a subsequence $\pi_{i_1}\cdots\pi_{i_k}$, where $1\le i_1<\cdots< i_k\le n$, such that $\pi_{i_j}<\pi_{i_m}$ if and only if $p_j<p_m$. For example, the permutation $31425$ has three occurrences of the pattern 123, namely, the subsequences 345, 145, and 125. Permutation patterns are a subject of great interest in the literature (e.g.\ see \cite{Kit5} and references therein). 

A \emph{partially ordered pattern} (\emph{POP}) $p$ of length $k$ is defined by a $k$-element partially ordered set (poset) $P$ labeled by the elements in $\{1,\ldots,k\}$. An occurrence of such a POP $p$ in a permutation $\pi=\pi_1\cdots\pi_n$ is a subsequence $\pi_{i_1}\cdots\pi_{i_k}$, where $1\le i_1<\cdots< i_k\le n$,  such that $\pi_{i_j}<\pi_{i_m}$ if and only if $j<m$ in $P$. Thus, a classical pattern of length $k$ corresponds to a $k$-element chain. For example, the POP $p=$ \hspace{-3.5mm}
\begin{minipage}[c]{3.5em}
\scalebox{1}{
\begin{tikzpicture}[scale=0.5]

\draw [line width=1](0,-0.5)--(0,0.5);

\foreach \x/\y in {0/-0.5,1/-0.5,0/0.5}
	\draw (\x,\y) node [scale=0.4, circle, draw, fill=black]{};

\node [left] at (0,-0.6){${\small 3}$};
\node [right] at (1,-0.6){${\small 2}$};
\node [left] at (0,0.6){${\small 1}$};

\end{tikzpicture}
}
\end{minipage}
occurs six times in the permutation $41523$, namely, as the subsequences $412$, $413$, $452$, $453$, $423$, and $523$. Clearly, avoiding $p$ is the same as avoiding the patterns $312$, $321$, and $231$ at the same time. 

Two sets of patterns, $S_1$ and $S_2$, are \emph{Wilf-equivalent} if the number of permutations of length $n$ avoiding each pattern in $S_1$ is equal to that avoiding each pattern in $S_2$ for any $n\ge 1$. In such a situation we write $S_1\sim S_2$.  If $\{p_1\}\sim\{p_2\}$ we simply write $p_1\sim p_2$.

Gao and Kitaev~\cite{GK19} made a number of conjectures regarding relations between permutations avoiding POPs and other combinatorial objects (including permutations avoiding classical patterns). Five conjectures, including one related to simultaneous avoidance of 8 patterns of length 4, were answered in \cite{YWZ}.  Five other conjectures, all related to Table 5 in \cite{GK19}  were confirmed in \cite{BNPU23,CL24}. The remaining conjecture in \cite[Table 5]{GK19}, related to the POP 
\begin{minipage}[c]{3em}
\scalebox{1}{
\begin{tikzpicture}[scale=0.4]
\draw [line width=1](0,0)--(1,1)--(1,2)--(1,3) (1,1)--(2,0);

\foreach \x/\y in {0/0,1/1,1/3,2/0,1/2}
	\draw (\x,\y) node [scale=0.3, circle, draw, fill=black]{};

\node [below] at (0,0){\small$3$};
\node [left] at (1,1){\small$5$};
\node [left] at (1,3){\small$2$};
\node [below] at (2,0){\small$4$};
\node [left] at (1,2){\small$1$};

\end{tikzpicture}
}
\end{minipage} 
and counted by \cite[\oeis{A224295}]{oeis}, in terms of regular (non-POP) patterns, is that 
\begin{equation}\label{main-W-eq} 
\{12345,12354\}\sim \{45123,45213\}.
\end{equation}
Bean et al.~\cite{BNPU23} used the Tilescope~\cite{ABCNPU22, BEM21} software package to generate the first 790 terms of the counting sequence of $\Av(12345,12354)$ and computed the first 50 terms of the counting sequence of $\Av(45123,45213)$ through different means. They also found 6 other classes, each avoiding two patterns of size 5, whose counting sequence matched \oeis{A224295} in the first 100 terms, and 6 more classes of pairs of patterns of size 5, whose counting sequence matched \oeis{A224295} in the first 20 terms \cite{Pan}. However, their attempt to solve conjecture~(\ref{main-W-eq}) by Tilescope led to a system of 23 equations with two catalytic variables, preventing the authors in \cite{BNPU23} from solving it.

The main result in this paper is proving~(\ref{main-W-eq}) and thus settling the final conjecture in  \cite[Table~5]{GK19}. To achieve our goals, we prove generalized versions of \emph{shape-\!Wilf-equivalence} (see Section~\ref{prelim-sec} for definitions) results of Backelin, West, and Xin in \cite{BWX07} (see Theorems~\ref{thm-direct-sum-sW} and~\ref{thm-direct-sum-sW-general}) and use a particular result on shape-\!Wilf-equivalence of monotone patterns in \cite{BWX07}. We also derive general results on shape-\!Wilf-equivalence of certain classes of POPs (see Theorems~\ref{thm-B5} and~\ref{thm-B6}) and use their specialization (also appearing in \cite{BE13}) as an essential piece in proving~(\ref{main-W-eq}). This specialization is the fact that the sets of patterns $\{123,213\}$ and $\{312,321\}$ are  shape-Wilf-equivalent on  bottom-left-justified $(0,1)$-Ferrers boards. Our results allow us to show (shape-)Wilf-equivalence of large classes of sets of patterns, including 11 out of 12 classes found by Bean et al.~\cite{BNPU23,Pan} in relation to the conjecture.

\section{Preliminaries}\label{prelim-sec}
For a permutation $\pi=\pi_1\pi_2\cdots\pi_n$, its \emph{reverse} is the permutation $r(\pi)=\pi^r=\pi_n\pi_{n-1}\cdots \pi_1$ and its \emph{complement} is the permutation $c(\pi)=\pi^c=c(\pi_1)c(\pi_2)\cdots c(\pi_n)$, where $c(x)=n+1-x$. Suppose $\pi=\pi_1\pi_2\cdots\pi_m \in S_m$ and $\sigma=\sigma_1\sigma_2\cdots\sigma_n\in S_n$. We define the \emph{direct sum} (or simply, \emph{sum}) $\oplus$ by building the permutation $\pi\oplus\sigma$ as follows:
\[
(\pi\oplus\sigma)_i =
\begin{cases}
\pi_i  &  \text{ if }1\le i\le m,\\
\sigma_{i-m}+m & \text{ if }m+1\le i\le m+n.
\end{cases}
\]
For example, $13425\oplus 2431=134257986$. 

Likewise, for a set of patterns $S$, let $S^r=\{\pi^r\mid \pi\in S\}$, $S^c=\{\pi^c\mid \pi\in S\}$, and $S^i=\{\pi^{-1}\mid \pi\in S\}$, where $\pi^{-1}$ is the usual group-theoretic inverse of $\pi$. Moreover, for a set of pattern $S$ and a pattern $\sigma$, let $S\oplus \sigma=\{\pi\oplus \sigma\mid \pi\in S\}$, and similarly for $\sigma\oplus S=\{\sigma\oplus \pi\mid \pi\in S\}$.

A \emph{Ferrers board} is a bottom-left-justified array of unit squares so that the number of squares in each row is less than or equal to the number of squares in the row below. To be precise, consider an $n \times n$ array of unit squares in the $xy$-plane, whose bottom left corner is at the origin $(0,0)$. The vertices of the unit squares are lattice points in $\mathbb{Z}^2$. For any vertex $V = (a, b)$, let $\Gamma(V )$ be the set of unit squares inside the rectangle $[0, a] \times [0, b]$. Then, a subset $F$ of the $n\times n$ array with the property that $\Gamma(V) \subseteq F$ for each vertex in $F$ is a Ferrers board. A square in a Ferrers board has coordinates $(i,j)$ if its corners have coordinates $(i-1,j-1)$, $(i-1,j)$, $(i,j-1)$ and $(i,j)$ for $i,j\ge 1$.

A \emph{filling} of a Ferrers board $F$ is an assignment of $0$'s and $1$'s such that every row and every  column of $F$ has exactly one 1 in it. Any such filling in a Ferrers board with $n$ columns naturally corresponds to a permutation $\pi_1\pi_2\ldots\pi_n$, where $\pi_i=j$ if there is a 1 in position $(i,j)$. For example,  in Figure~\ref{Ferr-diag-pic} we give a filling of a Ferrers board corresponding to the permutation 561423.

\begin{figure}
\[
\Yboxdim{14pt}\young(01,100,0001,000001,000010,001000)
\]
\vspace{-1.5\baselineskip}
\caption{A filling of a Ferrers board corresponding to the permutation 561423}\label{Ferr-diag-pic}
\end{figure}

Let $p$ be a pattern of length $k$ and $M(p)$ be the filling of the $k\times k$ array (entire square Ferrers board) corresponding to $p$. For a filling of a Ferrers board $F$, we say that it contains an occurrence of the pattern $p$ if there is a subset of $0$'s and $1$'s in $F$ forming $M(p)$ as a submatrix whose rows and columns belong entirely to $F$. For example, the filling in Figure~\ref{Ferr-diag-pic} avoids the pattern $312$ even though the permutation 561423 contains seven occurrences of this pattern. Indeed, for any such occurrence, for example formed by 514, there will be rows and columns not belonging entirely to the Ferrers board. On the other hand, the subsequence 123 of  561423 forms an occurrence of the pattern 123 in the Ferrers board in Figure~\ref{Ferr-diag-pic}.

Two sets of patterns, $S_1$ and $S_2$, are \emph{shape-\!Wilf-equivalent} if the number of fillings of any Ferrers board containing $n$ $1$'s and avoiding each pattern in $S_1$ is equal to that avoiding each pattern in $S_2$ for any $n\ge 1$. In such a situation we write $S_1\sim_s S_2$.  If $\{p_1\}\sim_s\{p_2\}$ we simply write $p_1\sim_s p_2$. We also denote by $F(S)$ (resp.,  $|F(S)|$) the set (resp., number) of fillings of a Ferrers board $F$ avoiding each pattern in a set of patterns $S$ simultaneously. $F(\{p\})$ is denoted by $F(p)$.

Clearly, shape-\!Wilf-equivalence implies Wilf-equivalence, which is shape-\!Wilf-equivalence on the entire square. In what follows, for two sets of patterns, $S_1$ and $S_2$, 
\[
S_1\oplus S_2:=\{s_1\oplus s_2\ |\ s_1\in S_1,s_2\in S_2\}. 
\]
The proof of the next theorem is essentially a copypaste, up to a rotation by 90 degrees counter-clockwise, of the proof of Proposition 2.3 in \cite{BWX07}, so we just sketch it below.

\begin{theorem}\label{thm-direct-sum-sW} 
Let $S$ and $S'$ be sets of patterns such that $S\sim_s S'$, and let $p$ be a pattern. Then $S\oplus\{p\}\sim_s S'\oplus\{p\}$.
\end{theorem}

\begin{proof} Suppose $F$ is a Ferrers board and a filling $f\in F(S\oplus\{p\})$. We do the following.
\begin{itemize}
\item For any square $(i,j)\in f$, if the subboard of $f$ that is above and to the right of $(i,j)$ contains an occurrence of $p$, then colour it red; otherwise, colour it blue.
\item Find the 1’s coloured blue, and colour the corresponding rows and columns blue.
\item  Remove the blue squares from $f$, denote by $F'$ the shape formed by the red squares, justified (squashed) bottom-left, and denote by $f'$ the induced filling of $F'$. Note that $F'$ is a Ferrers board (see the proof of Proposition 2.3 in \cite{BWX07} for justification) and the filling $f'$ avoids $S$.  
\item Since $S\sim_s S'$, we can map bijectively $f'$ into $f''$ having the same shape and avoiding $S'$. 
\item Reinsert the blue squares in the original places to obtain, in a bijective manner, a filling
$f'''\in F(S'\oplus\{p\})$.
\end{itemize}
Hence, $S\oplus\{p\}\sim_s S'\oplus\{p\}$.
\end{proof}

In fact, the observation, made right after the proof of Proposition 2.3 in \cite{BWX07}, works in our case as well. Namely,  the pattern $p$ in Theorem~\ref{thm-direct-sum-sW} can be replaced by a set of patterns essentially without any changes in the proof, so we have the following more general theorem, which yields the respective result in \cite{BWX07} for the singleton set. 

\begin{theorem}\label{thm-direct-sum-sW-general} Let $S,S',S''$ be sets of patterns such that $S'\sim_s S''$. Then
\[
S'\oplus S \sim_s S''\oplus S.
\]
\end{theorem}

\section{Shape-\!Wilf-equivalence of certain general classes of POPs}\label{POPs-sec}

In this section we discuss shape-\!Wilf-equivalence of certain general classes of POPs. A specialization of our results here will be used in the proof of our main result, the Wilf-equivalence~(\ref{main-W-eq}).

\begin{figure}[t]
  \centering
\begin{tikzpicture}[scale=0.8]

\draw [line width=1](0,0)--(1.5,1.5);
\draw [line width=1](1,0)--(1.5,1.5);
\draw [line width=1](3,0)--(1.5,1.5);

\draw (0,0) node [scale=0.4, circle, draw,fill=black]{};
\draw (1,0) node [scale=0.4, circle, draw,fill=black]{};
\draw (3,0) node [scale=0.4, circle, draw,fill=black]{};
\draw (1.5,1.5) node [scale=0.4, circle, draw,fill=black]{};

\draw (1.75,0) node [scale=0.15, circle, draw,fill=black]{};
\draw (2,0) node [scale=0.15, circle, draw,fill=black]{};
\draw (2.25,0) node [scale=0.15, circle, draw,fill=black]{};


\node [below] at (0,-0.1){$x_2$};
\node [below] at (1,-0.1){$x_3$};
\node [below] at (3,-0.1){$x_{k}$};
\node [above] at (1.5,1.5){$x_1$};
\end{tikzpicture}
\vspace{-0.3cm}
\caption{The form of POPs in Theorem~\ref{thm-B5}.}
 \label{pic-B5}
\end{figure}
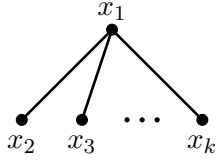

\begin{theorem}\label{thm-B5} 
Let $p_1$ and $p_2$ be any two POPs as in Figure~\ref{pic-B5}, where $\{x_1,x_2,\ldots,x_k\}=\{1,2,\ldots,k\}$ and $k\ge 1$. Then $p_1\sim_s p_2$. 
\end{theorem}

\begin{proof} Suppose that the element corresponding to $x_1$ in Figure~\ref{pic-B5} in $p_1$ (resp., $p_2$) is $i_1$ (resp., $i_2$).

We proceed by induction on the size (i.e.\ the number of squares) in Ferrers boards. Clearly the statement is true for the empty Ferrers board $F$ (of size 0) as in this case $F(p_1)=F(p_2)=\emptyset$. 

Suppose now that $F$ is a non-empty Ferrers board and consider a filling $f\in F(p_1)$.  Assume that the 1 in the top row in $f$ is in column $i$ and the top row has $\ell$ squares. Remove the top row in $f$ along with column $i$ to obtain the filling $f'$ avoiding $p_1$ on a Ferrers board of smaller size. Note that reinserting the top row in $f'$ and placing a 1 in it (so that a new column is also created) can be done in $\min\{k-1,\ell\}$ ways. Indeed, assuming $\ell\ge k$, creating a new column $i$ (by inserting a 1 in the top row) will result in an occurrence of $p_1$ unless $i\le i_1-1$ or $i\ge \ell-(k-i_1)+1$. Hence, there are $i_1-1+(\ell-(\ell-(k-i_1)+1)+1)=k-1$ valid ways to do this.  If $\ell<k$ then inserting a 1 in the top row can be done in any place.

Now, by the induction hypothesis, we can find, in a bijective way, a filling $f''$ avoiding the pattern $p_2$ corresponding to the filling $f'$ (in particular, $f'$ and $f''$ have the same Ferrers boards and hence the same number of squares in the top row). But then we can add a new top row with $\ell-1$ squares in $f''$ and then insert a new column, with a 1 in the top square, in the first $\ell-1$ places (so that the top row becomes of length $\ell$) in $\min\{k-1,\ell\}$ ways in order to avoid the pattern $p_2$. Indeed,  assuming $\ell\ge k$, creating a new row $i$ (by inserting a 1 in the top row) will result in an occurrence of $p_1$ unless $i\le i_2-1$ or $i\ge \ell-(k-i_2)+1$. If $\ell<k$, then 1 can be inserted anywhere in the top row. Finally, letting the choices to insert a 1 in the top row for $f'$ and $f''$ correspond to each other, say, from left to right, we obtain a bijective map showing that $p_1\sim_s p_2$. \end{proof}

Theorem~\ref{thm-B5} implies shape-\!Wilf-equivalence of pattern classes such as, for example, 
\[
\begin{split}
\{4123,4132,4213,4231,4312,4321\},\\
\{1423,1432,2413,2431,3412,3421\},\\
\{1243,1342,2143,2341,3142,3241\},\\
\{1234,1324,2134,2314,3124,3214\}.
\end{split}
\]

The proof of the next theorem is very similar to  the proof of Theorem~\ref{thm-B5}, so we only provide a sketch of it.

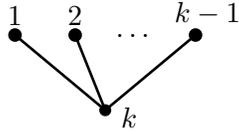
\begin{figure}[!h]
  \centering
\begin{tikzpicture}[scale=0.8]

\foreach \x in {0,1,3} 
   \draw [line width=1] (\x,2.5) node [scale=0.4, circle, draw, fill=black]{} -- (1.5,1.25);

\node at (2,2.5) {$\dots$};

\draw (1.5,1.25) node [scale=0.4, circle, draw, fill=black]{};


\node [above] at (0,2.5){$1$};
\node [above] at (1,2.5){$2$};
\node [above] at (3.2,2.5){$k-1$};
\node [above] at (1.9,0.8){$k$};
\end{tikzpicture}
\vspace{-0.3cm}
\caption{The form of one of the POP in Theorem~\ref{thm-B6}.}
 \label{pic-B6}
\end{figure}

\begin{theorem}\label{thm-B6} 
Let $p_1$ be a POP of the form in Figure~\ref{pic-B5} and $p_2$ be the POP in Figure~\ref{pic-B6}. Then $p_1\sim_s p_2$.
\end{theorem}

\begin{proof} By Theorem~\ref{thm-B5}, we may assume without loss of generality that $x_1=k$ in $p_1$.  Instead of removing the top row, as is done in the proof of  Theorem~\ref{thm-B5}, we remove the rightmost column, along with the row corresponding to the 1 in it. It is easy to see that reinserting the rightmost column and the row that intersects it, together with a 1 in the cell they share, results in an occurrence of the pattern $p_1$ (resp., $p_2$) unless the 1 is inserted in a square in the $k-1$ bottommost (resp., topmost) squares. Hence, we have $\min\{k-1,\ell\}$ valid insertion choices for either of the patterns, where $\ell$ is the number of squares in the rightmost column in the original shape. The rest is done by induction similarly to the proof of Theorem~\ref{thm-B5}. \end{proof}

As particular cases of Theorems~\ref{thm-B5} and~\ref{thm-B6}, we recover the known fact (see Table 3 in \cite{BE13}) that

\begin{center}
\begin{tabular}{ccccccccc}
\begin{minipage}[c]{3.5em}\scalebox{1}{
\begin{tikzpicture}[scale=0.5]

\draw [line width=1](0,-0.5)--(0.5,0.5)--(1,-0.5);

\foreach \x/\y in {0/-0.5,1/-0.5,0.5/0.5} 
	\draw (\x,\y) node [scale=0.4, circle, draw, fill=black]{};

\node [left] at (0,-0.6){${\small 2}$};
\node [right] at (1,-0.6){${\small 3}$};
\node [left] at (0.6,0.6){${\small 1}$};

\end{tikzpicture}
}\end{minipage}
 & $\sim_s$ & 
\hspace{-4mm} \begin{minipage}[c]{3.5em}\scalebox{1}{
\begin{tikzpicture}[scale=0.5]

\draw [line width=1](0,-0.5)--(0.5,0.5)--(1,-0.5);

\foreach \x/\y in {0/-0.5,1/-0.5,0.5/0.5} 
	\draw (\x,\y) node [scale=0.4, circle, draw, fill=black]{};

\node [left] at (0,-0.6){${\small 1}$};
\node [right] at (1,-0.6){${\small 3}$};
\node [left] at (0.6,0.6){${\small 2}$};

\end{tikzpicture}
}\end{minipage}
 & $\sim_s$ & 
\hspace{-4mm}  \begin{minipage}[c]{3.5em}\scalebox{1}{
\begin{tikzpicture}[scale=0.5]

\draw [line width=1](0,-0.5)--(0.5,0.5)--(1,-0.5);

\foreach \x/\y in {0/-0.5,1/-0.5,0.5/0.5} 
	\draw (\x,\y) node [scale=0.4, circle, draw, fill=black]{};

\node [left] at (0,-0.6){${\small 1}$};
\node [right] at (1,-0.6){${\small 2}$};
\node [left] at (0.6,0.6){${\small 3}$};

\end{tikzpicture}
}\end{minipage}
& $\sim_s$  & 
\hspace{-5mm}  \begin{minipage}[c]{3.5em}\scalebox{1}{
\begin{tikzpicture}[scale=0.5]

\draw [line width=1](0,0.5)--(0.5,-0.5)--(1,0.5);

\foreach \x/\y in {0/0.5,1/0.5,0.5/-0.5} 
	\draw (\x,\y) node [scale=0.4, circle, draw, fill=black]{};

\node [left] at (0.1,0.6){${\small 1}$};
\node [right] at (0.5,-0.6){${\small 3}$};
\node [left] at (1.85,0.6){${\small 2}$};

\end{tikzpicture}
}\end{minipage}

& $\sim_s$  & 
\hspace{-5mm}  \begin{minipage}[c]{3.5em}\scalebox{1}{
\begin{tikzpicture}[scale=0.5]

\draw [line width=1](0,0.5)--(0.5,-0.5)--(1,0.5);

\foreach \x/\y in {0/0.5,1/0.5,0.5/-0.5} 
	\draw (\x,\y) node [scale=0.4, circle, draw, fill=black]{};

\node [left] at (0.1,0.6){${\small 1}$};
\node [right] at (0.5,-0.6){${\small 2}$};
\node [left] at (1.85,0.6){${\small 3}$};

\end{tikzpicture}
}\end{minipage}

\end{tabular}
\end{center}
where the last shape-\!Wilf-equivalence follows from Theorem~\ref{thm:213-312} below (it is claimed to be proved in Section 6.1 in \cite{BE13}, but the proof contains inaccuracies in assuming that the shape contains some cells that may, in fact, be missing, so we provide an accurate proof here).

\begin{theorem} \label{thm:213-312} 
We have \begin{minipage}[c]{3.5em}\scalebox{1}{
\begin{tikzpicture}[scale=0.5]

\draw [line width=1](0,-0.5)--(0.5,0.5)--(1,-0.5);

\foreach \x/\y in {0/-0.5,1/-0.5,0.5/0.5} 
	\draw (\x,\y) node [scale=0.4, circle, draw, fill=black]{};

\end{tikzpicture}
}\end{minipage}\hspace{-5mm}  $\sim_s$ \hspace{-5mm}  \begin{minipage}[c]{3.5em}\scalebox{1}{
\begin{tikzpicture}[scale=0.5]

\draw [line width=1](0,0.5)--(0.5,-0.5)--(1,0.5);

\foreach \x/\y in {0/0.5,1/0.5,0.5/-0.5} 
	\draw (\x,\y) node [scale=0.4, circle, draw, fill=black]{};

\node [left] at (0.1,0.6){${\small 1}$};
\node [right] at (0.5,-0.6){${\small 2}$};
\node [left] at (1.85,0.6){${\small 3}$};

\end{tikzpicture}
}\end{minipage}
where the labelling of the first poset is arbitrary. \end{theorem}

\begin{proof} We proceed exactly in the same way as in the proof of Theorem~\ref{thm-B5} letting $k=3$ there and noting that for any POP of the form   \begin{minipage}[c]{3.5em}\scalebox{1}{
\begin{tikzpicture}[scale=0.5]

\draw [line width=1](0,-0.5)--(0.5,0.5)--(1,-0.5);

\foreach \x/\y in {0/-0.5,1/-0.5,0.5/0.5} 
	\draw (\x,\y) node [scale=0.4, circle, draw,fill=black]{};

\end{tikzpicture}
}\end{minipage}  \hspace{-5mm}   there are two choices to insert 1 in the top row if it is of length $\ge 2$; otherwise there is a unique choice.

To prove the theorem, we note that for the POP $p=$ 
\hspace{-5mm} 
\begin{minipage}[c]{3.5em}\scalebox{1}{
\begin{tikzpicture}[scale=0.5]

\draw [line width=1](0,0.5)--(0.5,-0.5)--(1,0.5);

\foreach \x/\y in {0/0.5,1/0.5,0.5/-0.5} 
	\draw (\x,\y) node [scale=0.4, circle, draw, fill=black]{};

\node [left] at (0.1,0.6){${\small 1}$};
\node [right] at (0.5,-0.6){${\small 2}$};
\node [left] at (1.85,0.6){${\small 3}$};

\end{tikzpicture}
}
\end{minipage}
\hspace{-1mm} 
we have the same number of choices to insert 1 in the top row. Indeed, this statement is obviously true if the top row is of length 1, so assume it is of length $\ge 2$. Consider the columns corresponding to the top row and the highest 1 in them below the top row. Suppose that this 1 is in column $c$.  It is easy to see that if 1 in the top row is in column $c'$ and $c'\not\in\{c-1,c+1\}$ then we obtain an occurrence of $p$ involving 1's in columns $c$ and $c'$ and a column between them.  

On the other hand, suppose that $c'=c-1$ (the case of $c'=c+1$ is analogous) and the 1 in column $c'$ is involved in an occurrence of $p$ formed by columns $x<y<z$. Clearly, $c'\ne y$ and $c\not\in\{x,y,z\}$. Suppose that $z=c'$ (the case of $x=c'$ is similar).  Then the 1's in the columns $c$, $x$ and $y$ also form an occurrence of $p$, which is below the top row, contradicting our choice of a $p$-avoiding filling below the top row in the inductive argument in the proof of Theorem~\ref{thm-B5}. Hence,  placing a 1 in the top row in columns $c-1$ or $c+1$ does not introduce an occurrence of $p$, and we always have two such choices matching two choices for the pattern 
\begin{minipage}[c]{3.5em}\scalebox{1}{
\begin{tikzpicture}[scale=0.5]

\draw [line width=1](0,-0.5)--(0.5,0.5)--(1,-0.5);

\foreach \x/\y in {0/-0.5,1/-0.5,0.5/0.5} 
	\draw (\x,\y) node [scale=0.4, circle, draw,fill=black]{};

\end{tikzpicture}
}\end{minipage}  \hspace{-7mm} . This ends the proof.\end{proof}

\begin{remark} \label{rem:why}
The part of this proof that corrects a statement in \cite{BE13} is that the highest 1 below the top row and in the columns of the top row need not be in the second highest row. The same modification corrects the same proof in \cite{BE13} in the cases of $\{132,213\}$ and $\{231,312\}$. Proceeding along the same argument as in the proof of Theorem~\ref{thm:213-312}, we choose $c'=1$ or $c'=c+1$ for $\{132,213\}$, and $c'=c-1$ or $c'=k$ (where $k$ is the length of the top row) for $\{231,312\}$.
\end{remark}

\section{Proof of the conjectured Wilf-equivalence (\ref{main-W-eq})} \label{sec:main}

Taking composition of reverse and complement operations yields $\{12345,12354\}\sim \{12345,21345\}$. Applying the reverse operation yields $\{45123,45213\}\sim\{32154,31254\}$. Hence, proving (\ref{main-W-eq}) is equivalent to proving
\begin{equation} \label{modif-main-W-eq} 
\{12345,21345\}\sim \{31254,32154\}.
\end{equation}

\begin{lemma}\label{lem1} 
We have $\{31245,32145\} \sim_s \{12345,21345\}$. 
\end{lemma}

\begin{proof} 
As a special case of Theorem~\ref{thm-B5} ($s=0$ and $x_1\in\{1,3\}$) we have $\{123,213\} \sim_s \{312,321\}$; this fact is also given in Table~3 in \cite{BE13}. But then, by Theorem~\ref{thm-direct-sum-sW}, $\{123,213\}\oplus 12 \sim_s \{312,321\}\oplus 12$, which completes the proof. 
\end{proof}

Note that in the next lemma we cannot claim shape-\!Wilf-equivalence, only Wilf-equivalence.

\begin{lemma}\label{lem2} 
We have $\{31245,32145\} \sim \{31254,32154\}$. 
\end{lemma}

\begin{proof} We use a special case of Proposition 2.2 in \cite{BWX07} stating that $12\sim_s 21$. Now, by Theorem~\ref{thm-direct-sum-sW-general} (for  $S=\{231,321\}$), we have
\[
\{12453,12543\}=12\oplus \{231,321\}\sim_s 21\oplus\{231,321\}=\{21453,21543\}.
\]
In particular, $\{12453,12543\}\sim\{21453,21543\}$. Taking the reverse and complement of each pattern, we obtain the desired result.
\end{proof}

The Wilf-equivalence in \eqref{modif-main-W-eq} now follows from Lemmas~\ref{lem1} and~\ref{lem2}.

\subsection{Related Wilf-equivalences} \label{subsec:more}

Bean et al.~\cite{BNPU23} mentions 6 more pairs of patterns of length 5 whose counting sequence was checked using Tilescope~\cite{ABCNPU22,BEM21} to coincide with \oeis{A224295} for at least 100 terms. Pantone~\cite{Pan} lists these 6 pairs and gives another 6 pairs of patterns of length 5 whose counting sequence was also checked to coincide with \oeis{A224295} for at least 20 terms. Here we prove that all but one ($\{13452,23451\}$) of these classes are Wilf-equivalent to $\{12345,12354\}$, i.e. enumerated by \oeis{A224295}.

\begin{corollary} \label{cor:more}
The following pattern sets are Wilf-equivalent to $\{12345, 12354\}$:
\[
\begin{split}
&\{12354, 12435\},\\
&\{12354, 12453\},\\
&\{12354, 21354\},\\
\end{split}
\qquad \qquad
\begin{split}
&\{12435, 12453\},\\
&\{12453, 12534\},\\
&\{12453, 12543\},\\
\end{split}
\qquad \qquad
\begin{split}
&\{12543, 21543\},\\
&\{13254, 21354\},\\
&\{13254, 23154\},\\
\end{split}
\qquad \qquad
\begin{split}
&\{21354, 21453\},\\
&\{21453, 21534\},\\
&\{21453, 21543\}.\\
\end{split}
\]
\end{corollary}

\begin{proof}
We have
\[
\begin{split}
\{12345, 12354\}&=12\oplus\{123,132\}=(\{123,213\}\oplus 12\})^{rc},\\
\{12354, 12435\}&=12\oplus\{132,213\}=(\{132,213\}\oplus 12\})^{rc},\\
\{12354, 12453\}&=12\oplus\{132,231\}=(\{213,312\}\oplus 12\})^{rc},\\
\{12354, 21354\}&=\{123,213\}\oplus 21,\\
\{12435, 12453\}&=12\oplus\{213,231\}=(\{132,231\}\oplus 12\})^{irc},\\
\{12453, 12534\}&=12\oplus\{231,312\}=(\{231,312\}\oplus 12\})^{rc},\\
\{12453, 12543\}&=12\oplus\{231,321\}=(\{312,321\}\oplus 12\})^{rc},\\
\{12543, 21543\}&=\{12,21\}\oplus 321,\\
\{13254, 21354\}&=\{132,213\}\oplus 21,\\
\{13254, 23154\}&=\{132,231\}\oplus 21,\\
\{21354, 21453\}&=21\oplus\{132,231\}=(\{213,312\}\oplus 21\})^{rc},\\
\{21453, 21534\}&=21\oplus\{231,312\}=(\{231,312\}\oplus 21\})^{rc},\\
\{21453, 21543\}&=21\oplus\{231,321\}=(\{312,321\}\oplus 21\})^{rc}.
\end{split}
\]
Then Wilf-equivalence of all these sets except $\{12543, 21543\}$ follows directly from Theorems~\ref{thm-B6} and \ref{thm:213-312} and Remark~\ref{rem:why}, Table 3 in \cite{BE13}, as well as a weaker version of Theorem~\ref{thm-direct-sum-sW-general} that derives Wilf-equivalence (rather than shape-Wilf-equivalence). For $\{12543, 21543\}=\{12,21\}\oplus 321$, we have 
\[
\{12,21\}\oplus 321=(321\oplus\{12,21\})^{rc}\sim(123\oplus\{12,21\})^{rc}=\{12,21\}\oplus 123=\{123,213\}\oplus 12,
\]
where the Wilf-equivalence follows from the fact that $321\sim_s 123$, a special case of Proposition 2.2 in \cite{BWX07}. 
\end{proof}

\section{Concluding remarks} \label{sec:conclusion}

The main result in our paper is proving~(\ref{main-W-eq}) (and hence settling the final conjecture in  \cite[Table 5]{GK19}), as well as restating in much more general terms certain known results in Theorems~\ref{thm-direct-sum-sW} and~\ref{thm-direct-sum-sW-general}.  Moreover, in Theorems~\ref{thm-B5} and~\ref{thm-B6} we open an interesting direction of studying shape-Wilf-equivalence of partially ordered patterns (POPs). 

We state the following conjecture generalizing the facts that  

\vspace{-0.2cm}

\begin{enumerate}
\item \begin{minipage}[c]{3.5em}\scalebox{1}{
\begin{tikzpicture}[scale=0.5]

\draw [line width=1](0.5,0.5)--(0.5,-0.5);

\foreach \x/\y in {0.5/0.5,0.5/-0.5} 
	\draw (\x,\y) node [scale=0.4, circle, draw, fill=black]{};

\node [right] at (0.5,-0.6){${\small 2}$};
\node [right] at (0.5,0.6){${\small 1}$};

\end{tikzpicture}
}\end{minipage}
\hspace{-5mm}$\sim_s$  
\begin{minipage}[c]{3.5em}\scalebox{1}{
\begin{tikzpicture}[scale=0.5]

\draw [line width=1](0.5,0.5)--(0.5,-0.5);

\foreach \x/\y in {0.5/0.5,0.5/-0.5} 
	\draw (\x,\y) node [scale=0.4, circle, draw, fill=black]{};

\node [right] at (0.5,-0.6){${\small 1}$};
\node [right] at (0.5,0.6){${\small 2}$};

\end{tikzpicture}
}\end{minipage} \hspace{-7mm}, a particular case of shape-Wilf-equivalence of monotone patterns in \cite{BWX07}, and   
\item \hspace{-5mm}  \begin{minipage}[c]{3.5em}\scalebox{1}{
\begin{tikzpicture}[scale=0.5]

\draw [line width=1](0,0.5)--(0.5,-0.5)--(1,0.5);

\foreach \x/\y in {0/0.5,1/0.5,0.5/-0.5} 
	\draw (\x,\y) node [scale=0.4, circle, draw, fill=black]{};

\node [left] at (0.1,0.6){${\small 1}$};
\node [right] at (0.5,-0.6){${\small 3}$};
\node [left] at (1.85,0.6){${\small 2}$};

\end{tikzpicture}
}\end{minipage}
$\sim_s$  
\hspace{-5mm}  \begin{minipage}[c]{3.5em}\scalebox{1}{
\begin{tikzpicture}[scale=0.5]

\draw [line width=1](0,0.5)--(0.5,-0.5)--(1,0.5);

\foreach \x/\y in {0/0.5,1/0.5,0.5/-0.5} 
	\draw (\x,\y) node [scale=0.4, circle, draw, fill=black]{};

\node [left] at (0.1,0.6){${\small 1}$};
\node [right] at (0.5,-0.6){${\small 2}$};
\node [left] at (1.85,0.6){${\small 3}$};

\end{tikzpicture}
}\end{minipage}; however, these patterns are not shape-Wilf-equivalent to \hspace{-5mm}  \begin{minipage}[c]{3.5em}\scalebox{1}{
\begin{tikzpicture}[scale=0.5]

\draw [line width=1](0,0.5)--(0.5,-0.5)--(1,0.5);

\foreach \x/\y in {0/0.5,1/0.5,0.5/-0.5} 
	\draw (\x,\y) node [scale=0.4, circle, draw, fill=black]{};

\node [left] at (0.1,0.6){${\small 2}$};
\node [right] at (0.5,-0.6){${\small 1}$};
\node [left] at (1.85,0.6){${\small 3}$};

\end{tikzpicture}
}\end{minipage}, which follows from Table 3 in \cite{BE13}.
\end{enumerate}

\begin{conjecture} For any $k\ge 2$,  
\hspace{-2mm}  \begin{minipage}[c]{3.5em}\scalebox{1}{
\begin{tikzpicture}[scale=0.5]

\foreach \x in {0,1,3} 
   \draw [line width=1] (\x,2.5) node [scale=0.4, circle, draw, fill=black]{} -- (1.5,1.25);

\node at (2,2.5) {$\dots$};

\draw (1.5,1.25) node [scale=0.4, circle, draw, fill=black]{};


\node [above] at (0,2.55){$1$};
\node [above] at (1,2.55){$2$};
\node [above] at (3.2,2.5){$k-1$};
\node [above] at (2.1,0.7){$k$};
\end{tikzpicture}}
\end{minipage}
\hspace{9mm}  $\sim_s$  
\hspace{-2mm}  \begin{minipage}[c]{3.5em}\scalebox{1}{
\begin{tikzpicture}[scale=0.5]

\foreach \x in {-1,0,2,3} 
   \draw [line width=1] (\x,2.5) node [scale=0.4, circle, draw, fill=black]{} -- (1.5,1.25);

\node at (1,2.5) {$\dots$};

\draw (1.5,1.25) node [scale=0.4, circle, draw, fill=black]{};


\node [above] at (-1,2.55){$1$};
\node [above] at (0,2.55){$2$};
\node [above] at (1.75,2.5){$k-2$};
\node [above] at (3,2.55){$k$};
\node [above] at (2.8,0.7){$k-1$};
\end{tikzpicture}}
\end{minipage} \hspace{10mm}  \begin{minipage}[c]{3.5em}\scalebox{1}.\end{minipage} 
\end{conjecture}

Finally, we state one more conjecture related to our paper.

\begin{conjecture}[Bean et al.\ \cite{BNPU23,Pan}] The $13$ pattern sets in Corollary~\ref{cor:more} are Wilf-equivalent to $\{13452,23451\}$ represented by the POP
\begin{minipage}[c]{3em}
\scalebox{1}{
\begin{tikzpicture}[scale=0.4]

\draw [line width=1] (0,0)--(1,1)--(1,2)--(1,3) (1,1)--(2,0);

\foreach \x/\y in {0/0,1/1,1/2,1/3,2/0}
	\draw (\x,\y) node [scale=0.3, circle, draw, fill=black]{};

\node [below] at (0,0){\small$1$};
\node [left] at (1,1){\small$2$};
\node [left] at (1,3){\small$4$};
\node [below] at (2,0){\small$5$};
\node [left] at (1,2){\small$3$};

\end{tikzpicture}
}
\end{minipage}. 
\end{conjecture}

\vspace{-0.5cm}

\section*{Acknowledgments} The authors are grateful to Joanna N.\ Chen for useful discussions related to our paper. The work of the fourth author was supported by the National Science Foundation of China (No. 12171362).

\end{document}